\documentclass[12pt]{amsart}
\usepackage{graphicx}
\usepackage{blkarray}
\usepackage{amsmath}
\usepackage{amsthm}
\usepackage{amssymb}
\usepackage{fullpage}
\usepackage{algorithmic}
\usepackage{multirow}
\usepackage[all]{xy}
\usepackage{longtable}
\usepackage{arydshln}
\usepackage{lscape}
\usepackage{rotating}
\usepackage{longtable}
\usepackage{caption}
\usepackage{float}
\usepackage{bm}
\usepackage{float}
\usepackage[makeroom]{cancel}
\usepackage{tikz}
\usetikzlibrary{arrows,decorations.markings}
\usepackage{setspace}
\usepackage{multirow}
\usepackage{appendix}
\usepackage[makeroom]{cancel}
\usepackage{color}
\usepackage{chngcntr}
\usepackage{wrapfig}
\usepackage{fullpage}
\usepackage{caption}
\usepackage{subcaption}
\usepackage{verbatim}
\usepackage{tikz}
\usepackage{multicol}
\usetikzlibrary{arrows,decorations.markings}
\usetikzlibrary{shapes}
\usetikzlibrary{matrix}
\usetikzlibrary{graphs}
\usepackage{arydshln}
\usetikzlibrary{backgrounds}
\usepackage{amsmath, amssymb}
\usepackage{algorithmic}

\usepackage{amssymb}
\usepackage{fullpage}
\counterwithout{table}{section}

\usepackage{algorithmic}
\usepackage{longtable}
\usepackage{caption}
\usepackage[makeroom]{cancel}

\usepackage{multicol}
\usepackage{multirow}

\usepackage{tikz}
\usetikzlibrary{arrows,decorations.markings}

\usepackage[pdftex,
colorlinks,
backref,
pagebackref,
hypertexnames=false,
linkcolor=blue,
urlcolor=green,
citecolor=red
]{hyperref}
\usepackage{bookmark}

\newcommand{\teq}{\trianglelefteq}

\theoremstyle{plain}
\newtheorem{theorem}{Theorem}[section]
\theoremstyle{plain}
\newtheorem{lemma}[theorem]{Lemma}
\theoremstyle{plain}

 \theoremstyle{remark}
\newtheorem{remark}[theorem]{Remark}

 \theoremstyle{remark}

\theoremstyle{definition}

\newtheorem{prop}[theorem]{Proposition}
\newtheorem{cor}[theorem]{Corollary}

\theoremstyle{theorem}

\newtheorem{conjB}[theorem]{Conjecture}

\numberwithin{equation}{section}

\def\DynkinArrowLength{3mm}

\def\C{\mathbb C}
\def\F{\mathbb F}

\def\Z{\mathbb Z}

\def\cF{\mathcal F}

\def\cN{\mathcal N}
\def\cP{\mathcal P}

\def\cS{\mathcal S}

\def\rA{\mathrm A}
\def\rB{\mathrm B}
\def\rC{\mathrm C}
\def\rD{\mathrm D}
\def\rE{\mathrm E}
\def\rF{\mathrm F}
\def\rG{\mathrm G}

\def\rY{\mathrm Y}

\def\Ind{\operatorname{Ind}}

\def\Inf{\operatorname{Inf}}
\def\Irr{\operatorname{Irr}}

\def\Tr{\operatorname{Tr}}
\def\cd{\operatorname{cd}}

\def\rs{\operatorname{rs}}

\title{On the character degrees of a Sylow $p$-subgroup of a finite Chevalley group $G(p^f)$ over a bad prime}
\author{Tung Le, Kay Magaard and Alessandro Paolini}

\begin{document}

\maketitle

\begin{abstract} Let $q$ be a power of a prime $p$ and let 
$U(q)$ be a Sylow $p$-subgroup of a finite Chevalley group $G(q)$ defined over the field with $q$ elements. 
We first give a parametrization of the set $\Irr(U(q))$ of irreducible characters 
of $U(q)$ when $G(q)$ is of type $\rG_2$. This is uniform for primes $p \ge 5$, 
while the bad primes $p=2$ and $p=3$ have to be considered separately. 
We then use this result and the contribution of several authors to show a general result, namely that 
if $G(q)$ is any finite Chevalley group with $p$ a bad prime, then there 
exists a character $\chi \in \Irr(U(q))$ such that $\chi(1)=q^n/p$ for some 
$n \in \Z_{\ge_0}$. In particular, for each $G(q)$ and every bad prime $p$, we construct a family of 
characters of such degree as inflation followed 
by an induction of linear characters of an abelian subquotient $V(q)$ of $U(q)$. 
\end{abstract}

\section{Introduction}
Let $q$ be a power of a prime $p$, and let $\F_q$ denote the field 
with $q$ elements. A major research problem in the representation 
theory of finite groups is to understand the characters of a finite 
Chevalley group $G(q)$ defined over $\F_q$. Namely finite Chevalley groups contribute to a 
large part of all finite nonabelian simple groups. The study of the set $\Irr(G(q))$ of 
ordinary irreducible characters of $G(q)$ has been carried out to an extensive progress, 
starting from the groundbreaking work of Deligne and Lusztig \cite{DL76}, 
to developments which allowed to compute and process the character 
table of $G(q)$ in \cite{CHEVIE} when the rank of $G(q)$ is small. In particular, 
the set $\cd(G(q))=\{\chi(1) \mid \chi \in \Irr(G(q))\}$ 
of irreducible character degrees of $G(q)$ is essentially known. 

The situation is different when we consider 
characters of $G(q)$ over an algebraically closed field of characteristic $\ell \ne p$, the 
so called cross-characteristics case. One has considerably less amount of 
information for such characters. The problem of studying such characters is 
closely related to the one of parametrizing the ordinary irreducible characters 
of a fixed Sylow $p$-subgroup $U(q)$ of $G(q)$, which we also denote by $U(G(q))$. Namely the induction of $\psi \in \Irr(U(q))$ 
to $G(q)$ remains an $\ell$-projective character, as $\ell$ and $p$ are different. A decomposition of 
such induced character can be provided, if the behavior of $\psi$ is known 
and we have enough information about the fusion of the conjugacy classes of $U(q)$ to $G(q)$. 

Even for groups of small rank, the set $\cd(U(q))$ is more 
complicated to describe than the set $\cd(G(q))$, and is in general 
also much bigger. 
We summarize here some 
of the main known results in this direction. If $G(q)=\rA_{n-1}(q)$ for $n \ge 2$, then 
every degree of a character in $\Irr(U(q))$ is a power of $q$ \cite{Is95}, and 
in fact $\cd(U(q))=\{1, q, \dots, q^{\mu(n)}\}$, where $\mu(n)=m^2-m$ if $n=2m$ and 
$\mu(n)=m^2$ if $n=2m+1$, see \cite{Hup} and \cite{Is07}. If $G(q)=\rD_n(q)$ for $n \ge 4$ and $p$ is odd, then 
$\cd(U(q))=\{1, q, \dots, q^{f(n)}\}$ with $f(n)=n(n-1)/2-[n/2]$ \cite{Mar99}.
A similar result holds for other $G(q)$ of classical type for odd $p$; in particular, $\cd(U(q)) \subseteq \{q^n \mid n \in \mathbb{Z}_{\ge 0}\}$ 
for $G(q)$ classical if and only if $p$ is an odd prime, see \cite{Sze03} and \cite{San03}. 
Via the Kirillov orbit method, it was proved in \cite{GMR15} that 
if $p$ is at least the Coxeter number of $G(q)$ then $\cd(U(q)) \subseteq \{q^n \mid n \in \mathbb{Z}_{\ge 0}\}$ 
for $G(q)$ an arbitrary finite Chevalley group. 

When $p$ is a bad prime for $G(q)$, the set $\cd(U(q))$ is often not known. 
The goal of this work is to show that $\cd(U(q)) \subseteq \{q^n \mid n \in \mathbb{Z}_{\ge 0}\}$ does 
never occur in this case, by means of an explicit construction of a character of degree $q^n/p$ for 
some positive integer $n$. 

We first determine a parametrization of $\Irr(U(q))$ when $G(q)$ is of 
type $\rG_2$. We provide full details just for the primes $p=2$ and $p=3$; 
the result is straightforward from \cite[Algorithm 3.3]{GLMP16} if $p \ge 5$. We 
use the subsequent Lemma \ref{lem:rl} to parametrize certain characters of $\Irr(U(q))$, 
and a counting argument to see that these determine all of $\Irr(U(q))$. In particular, we 
find characters of degree $q/p$ for $p \in \{2, 3\}$. 

\begin{theorem}\label{thm:G2} Let $G(q)=\rG_2(3^f)$ or $G(q)=\rG_2(2^f)$. The irreducible characters of $U(q)$ are parametrized in Table 
\ref{tab:G2all2}. In particular, we have that 
$$|\Irr(U(q))|=\begin{cases}
q^3+2q^2-q-1, & \text{ if }p\ge5, \\
2q^3+5q^2-10q+4, & \text{ if }p=3, \\
q^3+5q^2-7q+2, & \text{ if }p=2.
\end{cases}$$
\end{theorem}

It is not difficult to produce characters of degree $q/2$ in $U(\rB_2(q))$ when $q$ is a power of $2$, which we 
can inflate to $U(\rB_n(q))$, $U(\rC_n(q))$ and $U(\rF_4(q))$. We could similarly 
inflate characters of $U(\rD_4(2^f))$ of degree $q^3/2$, obtained in \cite{HLM11}, 
to $U(\rD_n(2^f))$ and $U(\rE_k(2^f))$, for $k=6, 7, 8$. For $q=3^f$, we have characters 
of degree $q^4/3$ in type $\rF_4(q)$ \cite{GLMP16} and of degree $q^7/3$ in 
type $\rE_6$ \cite{LM15} which we can inflate to $U(\rE_k(2^f))$ for $k=7, 8$. 
The work \cite{LM15} also gives an example of an irreducible character of 
$U(\rE_8(5^f))$ of degree $q^{16}/5$. Finally, the construction of characters 
of degrees $q/2$ in $\Irr(U(\rG_2(2^f)))$ and $q/3$ in $\Irr(U(\rG_2(3^f)))$ 
follows from Theorem \ref{thm:G2}.

This collection of results allows us to state the following. 

\begin{theorem}\label{theo:bad} Let $G(q)$ be a finite 
Chevalley group over $\F_q$.
If $p$ is a bad prime for $G(q)$, then there exist $\chi \in \Irr(U(q))$ and some $n \ge 1$ such that
$\chi(1)=q^n/p$. In particular, families of characters of such degree are constructed as an inflation, followed by an induction 
of a linear character of an abelian subquotient $V(q)$ of 
$U(q)$, with labels as in Table 
\ref{tab:badchar}. 
\end{theorem}
The labels of the characters are given as in \cite{GLMP16}. In general, a label of the form $a_i$ 
(respectively $b_j$) of $\chi \in \Irr(U(q))$ corresponds to an element of $\F_q^\times$ (respectively $\F_q$), 
which is the value on $x_i(1)$ (respectively  $x_j(1)$) of the 
linear character that we inflate and induce to obtain $\chi$. More details on these labels are given in the sequel for 
each case taken into consideration. 

The importance of the construction of such characters lies in the fact that these could replace some 
classes of characters, defined just for good $p$, helpful for investigating 
the cross-characteristics representations of $G(q)$. Let us for 
instance take $G(q)=\rD_4(q)$, and $\ell \ne p$ with $\ell \mid q+1$. The decomposition 
numbers are obtained in \cite{GP92} in the case when $p$ is an odd prime; this 
assumption is required for exploiting properties of the generalized Gelfand-Graev 
characters and the parametrization of Green functions in \cite{LS90}. 
A calculation shows that in the case 
$p=2$, by inducing to $G(q)$ the four irreducible characters of $U(q)$ 
corresponding to $(d_{1, 2, 4}, d_3) \in \F_2 \times \F_2$ in Proposition \ref{prop:D4char} 
we get characters that play the role of the 
$\ell$-projective characters $\Phi_6, \dots, \Phi_9$ in \cite[\S$5$]{GP92}, 
which turn out to be of major importance to determine the unitriangular 
shape of the decomposition matrix of $\rD_4(2^f)$.

We now examine the question of whether each of the families of irreducible characters of degree $q^n/p$ 
in Table \ref{tab:badchar} consists of all the 
characters of $\Irr(U(q))$ of such degree. It turns out that the previously mentioned works also 
determine that if $G(q)$ is not $\rE_8(5^f)$, then the families in Table \ref{tab:badchar} that do not arise 
from an embedding of a root system of smaller rank give in fact all the characters of $\Irr(U(q))$ of degree $q^n/p$ 
for the corresponding values of $n$. On the other hand, if a family $\cF$ does arise from an embedding of 
a smaller root system, then it is straightforward to 
get other characters in $\Irr(U(q))$ of the same degree which are not in $\cF$ 
by tensoring the characters in $\cF$ with linear characters of a certain root subgroup 
indexed by a root in the difference of the two root systems. 

We then propose the following conjecture. 

\begin{conjB} \label{conj:1}
The family in 
Table \ref{tab:badchar} of irreducible characters of degree $q^{16}/5$ 
in $\Irr(U(\rE_8(5^f)))$ consists of all irreducible characters of $U(\rE_8(5^f))$ whose degree is not 
a power of $q$. 
\end{conjB}

Finally, we present further progress and a question on fractional 
degrees in $U(q)$. The work \cite{GMP01} provides a construction 
of irreducible characters of fractional degrees with denominator of the 
form $p^t$ with $t \ge 2$. Namely if $G(q)=\rC_n(q)$ then there exist 
irreducible characters of $U(q)$ of degree $q^{n(n-1)/2}/2^{t}$ 
for every $0 \le t \le [n/2]$, and if $G(q)=\rD_n$ then there exist 
irreducible characters of $U(q)$ of degree $q^{3r(r-1)/2}/2^t$ 
(respectively  $q^{3r(r+1)/2}/2^t$) for 
every $1 \le t \le [r/2]$ if $n=2r$ (respectively if $n=2r+1$). 
On the one hand, such characters seem to maximize the 
power $m$ for character degrees of the form $q^m/p^t$.  
On the other hand, they seem not to maximize $t$, as 
in the case of $U(\rD_6(2^f))$ one just gets $t=1$ by applying the above formula, 
while by \cite{LMP17$^+$} we know that there also exist characters 
of the form $q^m/2^t$ with $t=2$ in $U(\rD_6(2^f))$. It would be 
interesting to determine, in general, all powers $q^m/p^t$ that 
can occur as character degrees for $\Irr(U(q))$, in particular 
the maximum value of $t$, for each finite Chevalley group $G(q)$.

\begin{table}[t]
\begin{center}
\begin{tabular}{|c|c|c|c|c|c|}
\hline
Type & Bad primes & Character labels & Size of family	& Degree & Unique\\
\hline
\hline
$\rB_n$ & $p=2$ & $\chi_{d_{n-1}, d_n}^{a_{2n-1}, a_{3n-2}}$ & $4(q-1)^2$	& $q/2$ & iff $n=2$\\
\hline
\hline
$\rC_n$ & $p=2$ & $\chi_{d_n, d_{n-1}}^{a_{2n-1}, a_{3n-2}}$ & $4(q-1)^2$	& $q/2$ & iff $n=2$\\
\hline
\hline
$\rD_n$ & $p=2$ & $\chi_{d_{1,2,4}, d_3}^{
a_{n+1, n+2, n+3}, a_{2n}, a_{2n+1}, a_{2n+2}}$ & $4(q-1)^4$	& $q^3/2$  & iff $n=4$\\
\hline
\hline
\multirow{2}{*}{$\rG_2$} & $p=2$ & $\chi_{d_1, d_2}^{a_3, a_4}$ 
& $4(q-1)^2$	& $q/2$ & yes \\
\cdashline{2-6}
& $p=3$ & $\chi_{e_1, e_2}^{a_3^*, a_5}$ 
& $9(q-1)^2/2$	& $q/3$ & yes \\
\hline
\hline
\multirow{2}{*}{$\rF_4$} & $p=2$ & $\chi_{d_2, d_3}^{a_6, a_9}$ 
& $4(q-1)^2$	& $q/2$ & no\\
\cdashline{2-6}
& $p=3$ & $\chi_{e_{1, 3, 4, 7}, e_2}^{a_{11}, a_{12}, a_{13}, a_6^*}$ 
& $9(q-1)^4/2$	& $q^4/3$ & yes \\
\hline
\hline
\multirow{2}{*}{$\rE_6$} & $p=2$ & $\chi_{d_{2, 3, 5}, d_4}^{
a_{8, 9, 10}, a_{13}, a_{14}, a_{15}}$ 
& $4(q-1)^4$	& $q^3/2$ & no\\
\cdashline{2-6}
& $p=3$ & $\chi_{e_{2, 1, 3, 5, 6, 7, 11}, e_4}^{a_{17}, a_{18}, a_{19}, a_{20}, a_{21}, a_{8, 9, 10}^*}$ 
& $9(q-1)^6/2$	& $q^7/3$ & yes\\
\hline
\hline
\multirow{2}{*}{$\rE_7$} & $p=2$ & $\chi_{d_{2, 3, 5}, d_4}^{
a_{9, 10, 11}, a_{15}, a_{16}, a_{17}}$ 
& $4(q-1)^4$	& $q^3/2$ & no\\
\cdashline{2-6}
& $p=3$ & $\chi_{e_{2, 1, 3, 5, 6, 8, 12}, e_4}^{a_{20}, a_{21}, a_{22}, a_{23}, a_{24}, a_{9, 10, 11}^*}$ 
& $9(q-1)^6/2$	& $q^7/3$ & no\\
\hline
\hline
\multirow{3}{*}{$\rE_8$} & $p=2$ & $\chi_{d_{2, 3, 5}, d_4}^{
a_{10, 11, 12}, a_{17}, a_{18}, a_{19}}$ 
& $4(q-1)^4$	& $q^3/2$ & no\\
\cdashline{2-6}
& $p=3$ & $\chi_{e_{2, 1, 3, 5, 6, 9, 13}, e_4}^{a_{23}, a_{24}, a_{25}, a_{26}, a_{27}, a_{10, 11, 12}^*}$ 
& $9(q-1)^6/2$	& $q^7/3$ & no\\
\cdashline{2-6}
& $p=5$ & $\chi_{f_{1 \dots, 4, 6, \dots, 11, 14, \dots, 17, 22, 23}, f_5}^{a_{37}, \dots, a_{43}, a_{12, 13}^*}$ 
& $25(q-1)^8/4$	& $q^{16}/5$ & Conj. \ref{conj:1} \\
\hline
\end{tabular}
\end{center}
\caption{Families of characters of $U(q)$ of degree $q^n/p$ for some $n \in \mathbb{N}$, and their uniqueness of such degree, for each bad prime $p$ and every Lie type.}
\label{tab:badchar}
\end{table}

\section{Preliminaries}

We first let $G$ be any finite group, $H$ be a subgroup of $G$, and 
$N$ be a normal subgroup of $G$. We recall some notation on characters 
of $G$ and its subgroups. 

We let $\Irr(G)$ be the set of irreducible characters of the group $G$. 
For a character $\chi \in \Irr(G)$, we denote by $\ker(\chi)$ the kernel of 
$\chi$ and by $Z(\chi)$ its centre. We denote by $\chi|_H$ the restriction 
of $\chi$ to $H$. Let $\varphi \in \Irr(G/N)$. Then we denote by 
$\Inf_{G/N}^G\varphi \in \Irr(G)$ the inflation of the character $\varphi$ to $G$. If 
$\psi$ is a character of $H$, then we denote by $\Ind_H^G(\psi)$ the induction 
of the character $\psi$ to $G$. We denote by $\langle \,\, , \, \rangle$ the usual 
inner product defined on the characters of $G$. If $\eta \in \Irr(H)$, then we denote 
$$\Irr(G \mid \eta)=\{\chi \in \Irr(G) \mid \langle \chi|_H, \eta \rangle \ne 0\}=\{\chi \in \Irr(G) \mid \langle \chi, \Ind_H^G\eta \rangle \ne 0\}.$$

We recall a result 
that we use 
several times in the sequel. The proof 
is a particular case of \cite[Lemma 2.1]{HLM16} 
detailed in \cite[$\S$4.1]{GLMP16}. 

\begin{lemma} \label{lem:rl}
Let $G$ be a finite group, let $H \le G$ and let $X$ be a transversal of $H$ in $G$.
Let $Y$ and $Z$ be subgroups of $H$, and $\lambda \in \Irr(Z)$.  Suppose that
\begin{itemize} \item[(i)] $Z \subseteq Z(G)$,
\item[(ii)] $X$ and $Y$ are elementary abelian groups 
with $|X|=|Y|$, 
\item[(iii)] $Y \trianglelefteq H$,
\item[(iv)] $Z \cap Y = 1$, and
\item[(v)] the commutator group $[X, Y]$ is contained in $Z$. 
\end{itemize}
If we put 
$$X':=\{x \in X \mid \lambda([x, y])=1 \text{ for all } y \in Y\}$$ 
and 
$$Y':=\{y \in Y \mid \lambda([x, y])=1 \text{ for all } x \in X\},$$
and if $\tilde Y$ is a complement of $Y'$ in $Y$, then the map 
\begin{equation}\label{eq:bijn1}
\Ind_{HX'}^{G} \Inf_{HX'/\tilde Y \ker \lambda}^{HX'}: \Irr(HX'/\tilde Y \ker \lambda \mid \lambda) \longrightarrow \Irr(G \mid \lambda).
\end{equation}
is a bijection.  
\end{lemma}

We keep the notation for $q$, $G(q)$, $U(q)$ and $\Irr(U(q))$ as in the Introduction. 
We briefly recall the notion of bad primes. 
Let $\Phi$ be the root system associated with $G(q)$, and let $\Phi^+$ be 
set of positive roots in $\Phi$. We fix an enumeration $\alpha_1, \dots, \alpha_{|\Phi^+|}$ 
of the positive roots, with $\alpha_1, \dots, \alpha_n$ the
simple roots of $\Phi^+$, 
and $\alpha_0:=\alpha_{|\Phi^+|}$ the highest root 
in $\Phi^+$. We say that $p$ is a \emph{bad prime} for $\Phi^+$ 
if $p$ divides one of the coefficients of $\alpha_0$ in its linear 
combination in terms of simple roots. 

We recall that as $G(q)$ is a split group, we have that 
$$U(q)=\prod_{\alpha \in \Phi^+} X_{\alpha}, \qquad \text{with} \qquad X_{\alpha}=\{x_{\alpha}(t) \mid t \in \F_q\} \cong (\F_q, +),$$
hence we have $|U(q)|=q^{|\Phi^+|}$. The group $X_{\alpha}$ is called the \emph{root subgroup} of $U(q)$ associated 
to $\alpha \in \Phi^+$, and each element $x_{\alpha}(t)$ is called the \emph{root element} with respect to 
$\alpha \in \Phi^+$ and $t \in \F_q$. 

We say $\cP \subseteq \Phi^+$ is a \emph{pattern} in $\Phi^+$ 
if for every $\alpha, \beta \in \cP$, either $\alpha+\beta \in \cP$ or $\alpha+\beta \notin \Phi^+$. 
For a pattern $\cP$, we have that the product 
$$X_{\cP}:=\prod_{\alpha \in \cP} X_{\alpha}$$
is well defined, and it is a subgroup of $U(q)$. We call $X_{\cP}$ the \emph{pattern group} 
corresponding to $\cP$. If $\cP:=\{\alpha_{i_1}, \dots, \alpha_{i_m}\}$, then we also 
write $X_{\{i_1, \dots, i_m\}}$ for $X_{\cP}$; similarly we write $x_i(t)$ for the root element $x_{\alpha_i}(t)$, with 
$\alpha_i \in \Phi^+$ and $t \in \F_q$. A subset $\cN$ of a pattern $\cP$ is \emph{normal} in $\cP$, or $\cN \teq \cP$, if 
for every $\alpha \in \cN$ and $\beta \in \cP$, one has either $\alpha+\beta \in \cN$ or 
$\alpha+\beta \notin \Phi^+$. It is easy to check that if $\cN \teq \cP$, then $X_{\cN} \teq X_{\cP}$. For $\chi \in \Irr(U(q))$, 
we define the \emph{central root support} $\rs(\chi):=\{\alpha \in \Phi^+ \mid X_{\alpha} \subseteq Z(\chi) \text{ and } X_{\alpha} \nsubseteq \ker(\chi)\}$. 

In order to construct the subquotients $V(q)$ as in Theorem \ref{theo:bad}, and to parametrize the corresponding characters, 
we need to fix a nontrivial character of $(\F_q, +)$. Denote by $\Tr : \F_q \to \F_p$ 
the field trace map. We define $\phi : \F_q \to \C^\times$
by $\phi(t) = e^{\frac{i2\pi \Tr(t)}{p}}$ for $t \in \F_q$. Notice that
\begin{equation} \label{eq:ker}
\ker{\phi}=\{t^p-t \mid t \in \F_q\}.\end{equation} 

\begin{remark} In the sequel, for each cyclic group $C$ of order $p$ we 
implicitly fix a morphism $\varphi:\Z/p\Z \to C$, and for 
each $d \in \Z/p\Z$ we denote by $\mu_C^d \in \Irr(C)$ the 
character such that $\mu_C^d(\varphi(1))=\zeta_p^d$, where 
$\zeta_p$ is a fixed primitive $p$-th root of unity. 
\end{remark}

\section{A parametrization of $\Irr(U(\rG_2(q)))$, when $q=2^f$ or $q=3^f$}

In this section we provide a parametrization of the irreducible 
characters of $U(q)$ when $G(q)=\rG_2(q)$ for every prime $p$. This is done by 
parametrizing families of characters of certain subquotients of $U(q)$, 
and checking by using the well-known formula 
\begin{equation}\label{eq:check}
|U(q)|=\sum_{\chi \in \Irr(U(q))} \chi(1)^2
\end{equation}
that these families give in fact all of $\Irr(U(q))$. We will denote by $T_1, T_2, \dots$ 
such subquotients of $\Irr(U(q))$ in the sequel. Each of the labels $\chi_{b_{j_1}, \dots, b_{j_s}}^{a_{i_1}, \dots, a_{i_r}}$ in Tables 
\ref{tab:badchar} and \ref{tab:G2all2}, with $a_{i_1}, \dots, a_{i_r} \in \F_q^\times$ and $b_{j_1}, \dots, b_{j_s} \in \F_q$, 
is obtained in a similar way as in \cite{GLMP16}, namely 
by inflation-induction process of the corresponding family of characters 
$$\lambda^{a_{i_1}, \dots, a_{i_r}} \otimes \mu_{j_1} \otimes \cdots \otimes \mu_{j_s} \in \Irr(V(q)).$$
The characters of $\Irr(U(\rG_2(q)))$ with labels of the form $d \in \F_2$ or $e \in \F_3$ are described in more detail in this section. 

We denote by $\alpha_1$ the long simple root in $\rG_2$, hence $\alpha_2$ is its short simple root. For every 
prime $p$, and for every $s, t \in \F_q$, the commutator relations among root elements are as follows, 
\begin{align*}
&\left[x_1(s), x_2(t)\right]=x_3(-st)x_4(st^2)x_5(-st^3)x_6(-s^2t^3), \\
&\left[x_2(s), x_3(t)\right]=x_4(2st)x_5(-3s^2t)x_6(3st^2), \qquad \left[x_2(s), x_4(t)\right]=x_5(3st), \\
&\left[x_1(s), x_5(t)\right]=x_6(st), \qquad\left[x_3(s), x_4(t)\right]=x_6(-3st),
\end{align*}
and $[x_i(s), x_j(t)]=1$ in the remaining cases. Observe that the 
irreducible characters of $U(\rG_2(q))$ when $p \ge 5$ can be easily parametrized by \cite[Algorithm 3.3]{GLMP16}. 

\begin{prop} Let $q=p^f$ with $p \ge 5$. Then $U(\rG_2(q))$ has exactly 
\begin{itemize}
\item[(i)] $q(q-1)$ irreducible characters of degree $q^2$, 
\item[(ii)] $q^2(q-1)+q(q-1)+(q-1)$ irreducible characters of degree $q$, and
\item[(iii)] $q^2$ linear characters.
\end{itemize}
\end{prop}

The characters of degree $q^2$ are precisely the ones with central root support $\{\alpha_6\}$, 
and each of the summands $q^i(q-1)$ in the expression for the number of irreducible characters 
of degree $q$ corresponds to the family of irreducible characters with central root support $\{\alpha_{i+3}\}$ 
for $i=0, 1, 2$. 

We examine next the case $q=2^f$. 

\begin{prop} Let $q=2^f$ and $G(q)=\rG_2(q)$. Then $U(q)$ has exactly 
\begin{itemize}
\item[(i)] $q(q-1)$ irreducible characters of degree $q^2,$
\item[(ii)] $(q-1)q^2+2(q-1)$ irreducible characters of degree $q,$ 
\item[(iii)] $4(q-1)^2$ irreducible characters of degree $q/2,$ and
\item[(iv)] $q^2$ linear characters.
\end{itemize}
\end{prop}

\begin{proof} Let $T_1:=U(q)$, and let $Z:=Z(T_1)=X_6$. Let us define $\lambda^{a_6}\in \Irr(Z)$ in 
the usual way. By the commutator relations, it is an easy check to deduce that the assumptions 
of Lemma \ref{lem:rl} are verified with $X:=X_1X_3$, $Y:=X_4X_5$ and $H:=X_2YZ$. We have that $X'=Y'=1$. Let $V_1(q):=X_2X_6$. Then the family 
$$\cF_1:=\{\Ind_{V_1(q)Y}^{U(q)}\Inf_{V_1(q)}^{V_1(q)Y}(\lambda^{a_6} \otimes \mu_2) \mid a_6 \in \F_q^\times, \mu_2 \in \Irr(X_2)\}$$
consists of $q(q-1)$ irreducible characters of $U(q)$ of degree $q^2$. 

Let now $T_2:=U(q)/X_6$. We have $Z:=Z(T_2)=X_5$. Again we apply Lemma \ref{lem:rl}; it is an easy check 
that its hypotheses are satisfied with $X:=X_2$, $Y:=X_4$ and $H:=X_1X_3YZ$. We have $X'=Y'=1$ also 
in this case. If $V_2(q):=X_1X_3X_5$, then we have that 
$$\cF_2:=\{\Ind_{V_2(q)Y}^{U(q)}\Inf_{V_2(q)}^{V_2(q)Y}(\lambda^{a_5} \otimes \mu_1 \otimes \mu_3) \mid 
a_5 \in \F_q^\times, \mu_1 \in \Irr(X_1), \mu_3 \in \Irr(X_3)\}$$
is a family of $q^2(q-1)$ irreducible characters of degree $q$. 

We now notice that $T_3:=U(q)/X_5X_6$ is isomorphic to $U(\rB_2(2^f))$ in the obvious way. 
By the subsequent Proposition \ref{prop:B2}, we get a family $\cF_3$ of $2(q-1)$ irreducible characters of degree 
$q$, a family $\cF_4$ of $4(q-1)^2$ irreducible characters of degree $q/2$, and the family 
$\cF_5$ of $q^2$ linear characters. 

Finally, notice that if $\chi_i$ is one of the characters in $\cF_i$, for $i=1, \dots, 5$, then we have 
\begin{align*}
\sum_{i=1}^{5}\chi_i(1)^2 |\cF_i|&=q^5(q-1)+q^4(q-1)+2q^2(q-1)+4q^2(q-1)^2/4+q^2=\\
&=q^6=|U(q)|,
\end{align*}
hence by Equation \eqref{eq:check} we have $\cF_1 \cup \dots \cup \cF_5=\Irr(U(q))$. \end{proof}

We now determine the irreducible characters of $U(\rG_2(q))$ when $q=3^f$. 

\begin{prop}
Let $q=3^f$ and $G(q)=\rG_2(q)$. Then $U(q)$ has
\begin{itemize}
\item[(i)]$(q-1)^2$ irreducible characters of degree $q^2$, 
\item[(ii)] $2(q-1)q^2+(q-1)(q+3)/2$ irreducible characters of degree $q$, 
\item[(iii)] $9(q-1)^2/2$ irreducible characters of degree $q/3$, and 
\item[(iv)] $q^2$ linear characters.
\end{itemize}
\end{prop}

\begin{proof} Let $T_1:=U(q)$. We have that $Z(T_1)=X_4X_6$. Let us 
put $Z:=X_6$ and let us define $\lambda^{a_6}$ as usual for $a_6 \in \F_q^\times$. 
Then Lemma \ref{lem:rl} applies with $X:=X_2$, $Y:=X_5$ and $H:=X_1X_3YZ$, and $X'=Y'=1$. 
Let $V_1(q):=X_1X_3X_4X_6$. Then $V_1(q) \cong (X_1X_3X_4) \times X_6=T' \times \F_q,$
where $T'$ is a special group of the form $q^{1+2}$ with $Z(T')=X_4$. One then has 
that $T'$ has $q-1$ irreducible characters of degree $q$ and $q^2$ linear characters. 
Hence we get two families of characters, namely 
$$\cF_1=\{\Ind_{X_{\{1,3,4,6\}}}^{U(q)}\Inf_{X_4X_6}^{X_{\{1,3,4,6\}}} (\lambda^{a_4, a_6}) \mid a_4, a_6 \in \F_q^\times\},$$
which consists of $(q-1)^2$ characters of $U(q)$ of degree $q^2$, and 
$$\cF_2=\{\Ind_{X_{\{1,3,4, 5,6\}}}^{U(q)}\Inf_{X_{\{1,3,6\}}}^{X_{\{1,3,4,5,6\}}} (\lambda^{a_6} \otimes \mu_1 \otimes \mu_3) \mid a_6 \in \F_q^\times, \mu_1 \in \Irr(X_1) \text{ and } \mu_3 \in \Irr(X_3)\},$$
which has $q^2(q-1)$ characters of $U(q)$ of degree $q$. 

Let us now define $T_2:=U(q)/X_6$. Then $X_4 \subseteq Z(T_2)$. We let $Z:=X_4$ and 
$\lambda^{a_4} \in \Irr(Z)$. It is a straightforward check that 
$X:=X_2$, $Y:=X_3$ and $H:=X_1X_5YZ$ satisfy the assumptions of Lemma \ref{lem:rl}. 
Again we have that $X'=Y'=1$. Notice that $V_2(q):=X_1X_4X_5X_6/X_6$ is 
an abelian group. Hence we get a family 
$$\cF_3=\{\Ind_{X_{\{1,3, 4,5,6\}}}^{U(q)}\Inf_{V_2(q)}^{X_{\{1,3, 4,5,6\}}} ( \lambda^{a_4} \otimes \mu_1\otimes \mu_5) \mid a_4 \in \F_q^\times, \mu_1 \in \Irr(X_1) \text{ and } \mu_5 \in \Irr(X_5)\}$$
of $q^2(q-1)$ irreducible characters of $U(q)$ of degree $q$. 

We now let $T_3:=U(q)/X_{\{4, 6\}}$. In this case, we have $Z:=Z(T_3)=X_{\{3, 5\}}$, 
and we define $\lambda^{a_3, a_5} \in \Irr(Z)$ for $a_3, a_5 \in \F_q$ in a similar way as in the case of $T_1$ and $T_2$. 
The groups $X:=X_1$, $Y:=X_2$ and $H:=X_{\{2, 3, 5\}}$ 
satisfy the hypotheses of Lemma \ref{lem:rl}. We now want to compute the sets $X'$ and $Y'$. We have that 
$$\lambda([x_2(t), x_1(s)])=\lambda(x_3(st)x_5(st^3))=\phi(st(a_3+a_5t^2)).$$

Let us first assume that $a_3, a_5 \in \F_q^\times$, and that $-a_3/a_5$ is a square. In this case, we write $a_3^*$ for $a_3$. 
Notice that there are $(q-1)^2/2$ such 
pairs of elements $a_3^*, a_5$ in $\F_q^\times$. 
Namely the set $S$ of squares in 
$\F_q^\times$ is a subgroup of $\F_q^\times$ of order $(q-1)/2$, and $-a_3^*/a_5$ is a square if 
and only if $a_5 \in \F_q^\times$ and $a_3^* \in -a_5S$. Let $\omega_{3, 5}$ be a fixed square root 
of $-a_3^*/a_5$. By Equation \eqref{eq:bijn1}, we have that 
$$X':=\{1, x_1(\pm 1/(a_3\omega_{3, 5}))\} \qquad \text{and} \qquad Y':=\{1, x_2(\pm \omega_{3, 5})\}.$$
In this case we have $[G:HX']=q/3$. Moreover, 
$V_3(q):=ZX'Y' \cong HX'/\tilde Y \ker \lambda$ 
is abelian. By Lemma \ref{lem:rl}, we obtain a family 
$$\cF_4:=\{\chi_{e_1, e_2}
^{a_3^*, a_5} \mid e_1, e_2 \in \F_3, a_5 \in \F_q^\times \text{ and } a_3^* \in -a_5S\},$$ 
where
$$\chi_{e_1, e_2}
^{a_3^*, a_5}:=\Ind_{HX'}^{U(q)} \Inf_{V_3(q)}^{HX'Y}(\lambda^{a_3^*, a_5} \otimes \mu_{X'}^{e_1} \otimes \mu_{Y'}^{e_2}),$$
of $9(q-1)^2/2$ irreducible characters of $U(q)$ of degree $q/3$. 

We now suppose that $a_3, a_5 \in \F_q^\times$, and $a_5/a_3$ is not a square. We write $a_3'$ for $a_3$. In this case, 
we have that $X'=Y'=1$. We put $V_4(q):=H/Y \cong X_3X_5$. We get a family
$$\cF_5:=\{\Ind_{HY}^{U(q)} \Inf_{V_4(q)}^{HY}(\lambda^{a_3', a_5})  \mid a_5 \in \F_q^\times \text{ and } a_3' \in \F_q^\times \setminus-a_5S\}$$
of $(q-1)^2/2$ irreducible characters of $U(q)$ of degree $q$. 

If exactly one of $a_3$ or $a_5$ is nonzero, then we also get $X'=Y'=1$. Let $V_5(q):=X_3X_5$. Then we get a family
$$\cF_6:=\{\Ind_{HY}^{U(q)} \Inf_{V_5(q)}^{HY}(\lambda^{a_3, a_5})  \mid (a_3, a_5) \in (\F_q^\times \times \{0\}) \cup (\{0\} \times \F_q^\times)\}$$
of $2(q-1)$ irreducible characters of degree $q$ of $U(q)$. 
The choice $a_3=a_5=0$ corresponds to the family $\cF_7$ of $q^2$ linear characters of $U(q)$.  

Finally, if $\chi_i$ is any character in $\cF_i$, for $i=1, \dots, 7$, then we have 
\begin{align*}
\sum_{i=1}^{7}\chi_i(1)^2 |\cF_i|&=q^4(q-1)^2+q^2(2q^2(q-1)+(q-1)(q+3)/2)+q^2(q-1)^2/2+q^2=\\
&=q^6=|U(q)|,
\end{align*}
hence $\cF_1 \cup \dots \cup \cF_7=\Irr(U(\rG_2(q)))$. \end{proof}

\begin{table}
\begin{center}
\begin{tabular}{|c|c|c|}
\hline
\multicolumn{3}{|c|}{$\Irr(U(\rG_2(2^f)))$} \\
\hline
\hline
Labels & Size of family & Deg. \\
\hline
\hline
$\chi_{b_2}^{a_6}$ & $q(q-1)$	& $q^2$ \\
\hline 
$\chi_{b_1, b_3}^{a_5}$ & $q^2(q-1)$	& $q$ \\
\hline 
$\chi^{a_3, 0}$ & $q-1$	& $q$ \\
\hline 
$\chi^{0, a_4}$ & $q-1$	& $q$ \\
\hline 
$\chi_{d_1, d_2}^{a_3, a_4}$ & $4(q-1)^2$	& $q/2$ \\
\hline 
$\chi_{b_1, b_2}$ & $q^2$	& $1$ \\
\hline 
\multicolumn{3}{c}{}\\
\multicolumn{3}{c}{}\\
\end{tabular}
\quad\quad\quad\quad\quad
\begin{tabular}{|c|c|c|}
\hline
\multicolumn{3}{|c|}{$\Irr(U(\rG_2(3^f)))$} \\
\hline
\hline
Labels & Size of family & Deg. \\
\hline
\hline
$\chi^{a_4, a_6}$ & $(q-1)^2$	& $q^2$ \\
\hline
$\chi_{b_1, b_3}^{a_6}$ & $q^2(q-1)$	& $q$ \\
\hline
$\chi_{b_1, b_5}^{a_4}$ & $q^2(q-1)$	& $q$ \\
\hline
$\chi^{a_3', a_5}$ & $(q-1)^2/2$	& $q$ \\
\hline
$\chi^{a_3, 0}$ & $q-1$	& $q$ \\
\hline
$\chi^{0, a_5}$ & $q-1$	& $q$ \\
\hline
$\chi_{e_1, e_2}^{a_3^*, a_5}$ & $9(q-1)^2/2$	& $q/3$ \\
\hline
$\chi_{b_1, b_2}$ & $q^2$	& $1$ \\
\hline
\end{tabular}
\end{center}
\caption{A parametrization of $\Irr(U(\rG_2(q)))$ for $p=2$ and $p=3$.}
\label{tab:G2all2}
\end{table}

\section{Characters of fractional degree of $U(q)$ in classical type}

We now focus on the characters of $U(\rB_2(q))$ when $q=2^f$. The family of 
characters of degree $q/2$ in $U(\rB_2(q))$ was obtained  
in \cite[\S7]{Lus03} and revisited in \cite[\S7]{BD06} 
in the context of character sheaves. We construct it here as 
an inflation-induction process from some subquotient of $U(\rB_2(q))$. 

\begin{prop} \label{prop:B2}
Let $q=2^f$. Then there are exactly $4(q-1)^2$ irreducible characters 
of $U(\rB_2(q))$ of degree $q/2$. 
\end{prop}

\begin{proof} Since $p=2$, we have 
that $[x_3(s), x_2(t)]=1$. 
Hence $Z:=Z(U(q))=X_{\{3, 4\}}$. 
Let 
$X:=X_1$, $Y:=X_2$ and $H:=X_{\{2, 3, 4\}}$. Then the 
assumptions 
of Lemma \ref{lem:rl} are 
satisfied. Let us fix $a_3, a_4 \in \F_q^\times$, and let  $\lambda=\lambda^{a_3, a_4} \in \Irr(Z)$ be such that 
$\lambda(x_i(t))=\phi(a_it)$ for $i=3, 4$. 
We have that 
$$\lambda([x_1(s), x_2(t)])=\lambda(x_3(st)x_4(st^2))=
\phi(st(a_3+a_4t)),$$
and in the notation of Lemma \ref{lem:rl}, by Equation \eqref{eq:ker} we have 
$$X':=\{1, x_1(a_4/a_3^2)\} \qquad \text{and} \qquad 
Y':=\{1, x_2(a_3/a_4)\}.$$
Notice that $[U(q):HX']=q/2$ and that 
$HX'/\tilde Y \ker \lambda\cong ZX'Y'$. Let us 
define $V(q):=ZX'Y'$. By 
Equation \eqref{eq:bijn1} we have that 
$$\Ind_{HX'}^{U(q)} \Inf_{V(q)}^{HX'}:\Irr(V(q) \mid \lambda) \longrightarrow \Irr(G \mid \lambda)$$
is a bijection. Moreover, we have that 
$V(q)$ is abelian, hence 
$$\Irr(V(q) \mid \lambda^{a_3, a_4})=\{\psi_{d_1, d_2}^{a_3, a_4} 
\mid d_1, d_2 \in \F_2\},$$
where 
$$\psi_{d_1, d_2}^{a_3, a_4}:=\lambda^{a_3, a_4} \otimes 
\mu_{X'}^{d_1} \otimes 
\mu_{Y'}^{d_2}.$$
Notice that the sets $\Irr(G \mid \lambda^{a_3, a_4})$ are disjoint for $a_3, a_4 \in \F_q^\times$. 
 
Finally, recall by \cite[\S7]{Lus03} that the other characters in $\Irr(U(q))$ consist of two families of size $q-1$ of characters of degree $q$, 
namely the irreducible characters with central root support $\{\alpha_3\}$ and $\{\alpha_4\}$ respectively, and the family of 
the $q^2$ linear characters. 
\end{proof}

\begin{figure}[h]
\begin{center}
\begin{tikzpicture}[transform shape, place/.style={circle,draw=black,fill=black, tiny},middlearrow/.style={
    decoration={markings,
      mark=at position 0.6 with
      {\draw (0:0mm) -- +(+135:\DynkinArrowLength); \draw (0:0mm) -- +(-135:\DynkinArrowLength);},
    },
        postaction={decorate}
  },
    middlearrowtwo/.style={
   decoration={markings,
      mark=at position 0.5 with
      {\draw +(-45:\DynkinArrowLength) -- (0:0mm); \draw +(+45:\DynkinArrowLength) -- (0:0mm);},
   },
     postaction={decorate}
  },
 dedge/.style={
    middlearrow,
    double distance=0.5mm,
  },
   dedgetwo/.style={
    middlearrowtwo,
    double distance=0.5mm,
  }]]
    \node(z) at (7, 0)[label=right:\begin{large}$\rB_n$\end{large}]  {}; 
  \node(a) at (-5,0) [circle, draw, thick, fill=none, inner sep=2pt,label=below:$\alpha_1$] {};
  \node (b) at (-3,0) [circle, draw, thick, fill=none, inner sep=2pt,label=below:$\alpha_2$] {};
  \node (c) at (-1,0) [circle, draw, thick, fill=none, inner sep=2pt,label=below:$\alpha_3$] {};
  \node (d) at (0,0) {};
    \node (e) at (2,0) {};
  \node (f) at (3,0) [circle, draw, thick, fill=none, inner sep=2pt,label=below:$\alpha_{n-1}$] {};
    \node (g) at (5,0) [circle, draw, thick, fill=none, inner sep=2pt,label=below:$\alpha_n$] {};
  \draw (a) -- (b);
  \draw (e) -- (f);
    \draw[dashed] (d) -- (e);
  \draw (c) -- (d);
  \draw (b) -- (c) ;
  \path (f) edge[dedge] (g);
\end{tikzpicture}
\end{center}

\begin{center}
\begin{tikzpicture}[transform shape, place/.style={circle,draw=black,fill=black, tiny},middlearrow/.style={
    decoration={markings,
      mark=at position 0.6 with
      {\draw (0:0mm) -- +(+135:\DynkinArrowLength); \draw (0:0mm) -- +(-135:\DynkinArrowLength);},
    },
        postaction={decorate}
  },
    middlearrowtwo/.style={
   decoration={markings,
      mark=at position 0.5 with
      {\draw +(-45:\DynkinArrowLength) -- (0:0mm); \draw +(+45:\DynkinArrowLength) -- (0:0mm);},
   },
     postaction={decorate}
  },
 dedge/.style={
    middlearrow,
    double distance=0.5mm,
  },
   dedgetwo/.style={
    middlearrowtwo,
    double distance=0.5mm,
  }]]

    \node(z) at (7, 0)[label=right:\begin{large}$\rC_n$\end{large}]  {}; 
      \node(a) at (-5,0) [circle, draw, thick, fill=none, inner sep=2pt,label=below:$\alpha_1$] {};
  \node (b) at (-3,0) [circle, draw, thick, fill=none, inner sep=2pt,label=below:$\alpha_2$] {};
  \node (c) at (-1,0) [circle, draw, thick, fill=none, inner sep=2pt,label=below:$\alpha_3$] {};
  \node (d) at (0,0) {};
    \node (e) at (2,0) {};
  \node (f) at (3,0) [circle, draw, thick, fill=none, inner sep=2pt,label=below:$\alpha_{n-1}$] {};
    \node (g) at (5,0) [circle, draw, thick, fill=none, inner sep=2pt,label=below:$\alpha_n$] {};
  \draw (a) -- (b);
  \draw (e) -- (f);
    \draw[dashed] (d) -- (e);
  \draw (c) -- (d);
  \draw (b) -- (c) ;
  \path (f) edge[dedgetwo] (g);
\end{tikzpicture}
\end{center}
\caption{The Dynkin diagrams of $\rB_n$ and $\rC_n$. Simple roots are labelled as in CHEVIE.}
 \label{fig:Dynkin}
\end{figure}
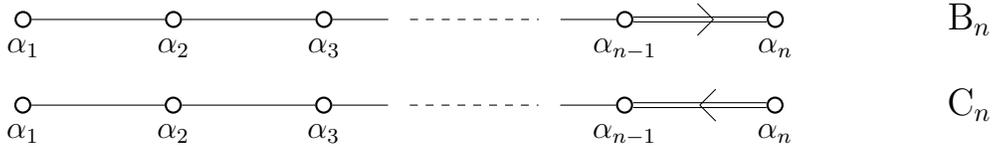

By inflation of the irreducible characters in Proposition \ref{prop:B2} 
we obtain the following. 

\begin{cor} \label{co:BC} Let $G(q)=\rB_n(q)$ or $G(q)=\rC_n(q)$. Then 
$U(q)$ has at least $4(q-1)^2$ characters of degree $q/2$.
\end{cor}

\begin{proof} Let us define 
$$\cN:= \bigcup_{i=1}^{n-2}\{\alpha \in \Phi^+ \mid \alpha_i \le \alpha\}.$$
We have that $\cN \teq \Phi^+$, hence $X_{\cN} \teq U(q)$, and $U(q)/X_{\cN} \cong X_{\{n-1, n, 2n-1, 3n-2\}}$ 
is isomorphic to $U(\rB_2(q))$, with $\alpha_{n-1}$ long and $\alpha_n$ short in type 
$\rB_n$ and viceversa in type $\rC_n$, in the notation of Figure \ref{fig:Dynkin}. If $a_{2n-1}, a_{3n-2} \in \F_q^\times$ and 
$d_{n-1}, d_n \in \F_2$, we then define 
$$\chi_{d_{n-1}, d_n}^{a_{2n-1}, a_{3n-2}}:=\Inf_{U(q)/X_{\cN}}^{U(q)} \psi_{d_{n-1}, d_n}^{a_{2n-1}, a_{3n-2}}$$
if $U(q)=U(\rB_2(q))$, and 
$$\chi_{d_n, d_{n-1}}^{a_{2n-1}, a_{3n-2}}:=\Inf_{U(q)/X_{\cN}}^{U(q)} \psi_{d_n, d_{n-1}}^{a_{2n-1}, a_{3n-2}}$$
if $U(q)=U(\rC_2(q))$, where $\psi_{d_{n-1}, d_n}^{a_{2n-1}, a_{3n-2}}$ and 
$\psi_{d_n, d_{n-1}}^{a_{2n-1}, a_{3n-2}}$ are defined as in Proposition \ref{prop:B2}. 
\end{proof}

Let us now examine the groups of type $\rD_4$. The irreducible characters 
of $U(\rD_4(q))$ have been completely parametrized in \cite{HLM11} 
for every prime $p$. Unlike the case of type $\rB_n$ and $\rC_n$, there 
are no characters of degree $q/2$ in type $\rD_4$. 
The statement and the construction below combine the study of 
the family $\cF_{8, 9, 10}^{\text{even}}$ in \cite{HLM11} and the approach 
of \cite{GLMP16}. 

\begin{prop}\label{prop:D4char} Let $G(q)=\rD_4(q)$. Then $U(q)$ has exactly $4(q-1)^4$ irreducible 
characters of degree $q^3/2$. 
\end{prop}

\noindent
\emph{Construction.} Let $\cN:=X_{\{11, 12\}}$. Notice that $\cN \teq \Phi^+$. For fixed 
$a_8, a_9, a_{10} \in \F_q^\times$, define 
$$x_{1, 2, 4}(t):=x_1(a_{10}t)x_2(a_9t)x_4(a_8t) \quad \text{and} \quad x_{5, 6, 7}(s):=x_5(a_{10}s)x_6(a_9s)x_7(a_8s)$$
for every $s, t \in \F_q$. Let
$$X':=\{x_{1, 2, 4}(t) \mid t \in \F_q\} \qquad \text{and} \qquad Y':=\{x_{5, 6, 7}(s) \mid s \in \F_q\},$$
and let $\tilde{Y}:=X_5X_6$. Then $Y = Y' \times \tilde{Y}$.

Let us also fix $a_{5,6,7} \in \F_q^\times$. We define $\lambda=\lambda^{a_{5, 6, 7}, a_8, a_9, a_{10}}$ by $\lambda(x_{5, 6, 7}(s))=\phi(a_{5,6,7}s)$, and 
$\lambda(x_i(t_i))=\phi(a_it_i)$ for $i=8, 9, 10$, and we define 
$$W_1:=\{1, x_{1,2,4}(a_{5,6,7}/(a_8a_9a_{10})\} \quad \text{and} \quad W_2:=\{1, x_3(a_8a_9a_{10}/a_{5,6,7}^2)\}$$
and $V(q):=W_1W_2YZ/(\tilde{Y}\ker \lambda)$. Then each character of
$$\cF:=\{\psi_{d_{1,2,4}, d_3}^{a_{5, 6, 7}, a_8, a_9, a_{10}} \mid a_{5, 6, 7}, a_8, a_9, a_{10} \in \F_q^\times, 
d_{1,2,4}, d_3 \in \F_2\},$$
where
$$\psi_{d_{1,2,4}, d_3}^{a_{5, 6, 7}, a_8, a_9, a_{10}}:=\Ind_{X'W_2YX_{\{8, \dots, 12\}}}^{U(q)}
\Inf_{V(q)}^{X'W_2YX_{\{8, \dots, 12\}}}(\lambda^{a_{5, 6, 7}, a_8, a_9, a_{10}} \otimes \mu_{W_1}^{d_{1,2,4}}
 \otimes \mu_{W_2}^{d_3}),$$
is irreducible in $U(q)$ of degree $q^3/2$. The characters $\psi_{d_{1,2,4}, d_3}^{a_{5, 6, 7}, a_8, a_9, a_{10}}$ are all distinct. 

Finally, by \cite{HLM11}, there are no other characters in $\Irr(U(q))$ of degree $q^n/2$ for any $n \ge 0$. 
\hfill\ensuremath{\square}

\vspace{2mm}

As done in the case of type $\rB_n$ and $\rC_n$, we obtain by inflation 
characters of degree $q^3/2$ in type $\rD_n$ for every $n \ge 5$.

\begin{cor} For $n \ge 5$, the group $U(\rD_n(q))$ has $4(q-1)^4$ irreducible 
characters of degree $q^3/2$. 
\end{cor}

\begin{proof} In a similar way as in Corollary \ref{co:BC}, we have that 
$$\cN:= \bigcup_{i=5}^{n}\{\alpha \in \Phi^+ \mid \alpha_i \le \alpha\} \teq \Phi^+.$$
Then we have that 
$$U(q)/X_{\cN}=X_{\cS} \cong U(\rD_4(q)),$$ 
where $\cS:=\{1, 2, 3, 4, n+1, n+2, n+3, 2n, 2n+1, 2n+2, 3n-1, 4n-3\}$. 
We then apply Proposition \ref{prop:D4char}, namely if $a_{n+1, n+2, n+3}, a_{2n}, a_{2n+1}, a_{2n+2} \in \F_q^\times$ and 
$d_{1, 2, 4}, d_3 \in \F_2$, the characters as in the claim are given by 
$$\chi_{d_{1,3,4}, d_2}^{a_{n+1, n+2, n+3}, a_{2n}, a_{2n+1}, a_{2n+2}}:=
\Inf_{U(q)/X_{\cN}}^{U(q)}\psi_{d_{1,3,4}, d_2}^{a_{n+1, n+2, n+3}, a_{2n}, a_{2n+1}, a_{2n+2}},$$
where each of the $\psi_{d_{1,3,4}, d_2}^{a_{n+1, n+2, n+3}, a_{2n}, a_{2n+1}, a_{2n+2}}$ is defined as in 
Proposition \ref{prop:D4char}. \end{proof}

\section{Characters of fractional degree of $U(q)$ in types $\rF_4$ and $\rE_k$}

We now move on to character degrees of the form 
$q^n/p$ in type $\rF_4$. We first consider 
the case of the prime $p=2$. We have 
$$\cN:= \{\alpha \in \Phi^+ \mid \alpha_1 \le \alpha\} \cup \{\alpha \in \Phi^+ \mid \alpha_4 \le \alpha\} \teq \Phi^+,$$
and $U(q)/X_{\cN} \cong 
U(\rB_2(q))$. Hence we obtain characters of degrees $q/2$ 
in $U(\rF_4(q))$. 
For $p=3$, we have the following explicit construction. 

\begin{prop}[\cite{GLMP16}, $\S 4.3$] Let $q=3^f$, and let $G(q)=\rF_4(q)$. There exist exactly 
$9(q-1)^4/2$ irreducible characters of $U(q)$ of degree $q^4/3$.
\end{prop}

\noindent
\emph{Construction.} Let $\cN:=\{\alpha_{14}, \dots, \alpha_{24}\}$. Then $\cN \teq \Phi^+$. For $a_{11}, 
a_{12}, a_{13} \in \F_q^\times$, let 
$$x_{1, 3, 4, 7}(t):=x_1(a_{13}t)x_1(a_{13}t)x_3(a_{-12}t)x_4(a_{11}t)
x_7(-a_{11}a_{12}t^2)$$
for every $t \in \F_q$. Moreover, let $S$ be the set of squares in $\F_q^\times$, and for every 
$a_6^* \in \F_q^\times$ such that $a_6^*/a_{11}a_{12}a_{13} \in S$ 
let $e$ be a square root of $a_6^*/a_{11}a_{12}a_{13}$, and let 
$$X':=\{x_{1, 3, 4, 7}(t) \mid t \in \F_q\}$$ 
and 
$$W_1:=\{x_{1, 3, 4, 7}(es) \mid s \in \F_3\}, \qquad 
W_2:=\{x_2(t/(a_{11}a_{12}^2a_{13}e^3)) \mid t \in \F_3\}.$$

Let $\lambda:=\lambda^{a_{11},a_{12}, a_{13}, a_6^*}$ 
be defined as $\lambda(x_i(t))=\phi(a_it)$ for $i=11, 12, 13$ 
and $ \lambda(x_6(t))=\phi(a_6^*t)$. Let $Y:=X_{\{5, 8, 9, 10\}}$, and let 
$V(q):=W_1W_2X_6YZ/\ker(\lambda)$. 
Then we have that
$$\cF:=\{\chi_{e_{1, 3, 4, 7}, e_2}^{a_{11}, a_{12}, a_{13}, a_6^* } \mid a_{11}, a_{12}, a_{13} \in \F_q^\times, e_{1, 3, 4, 7}, e_2 \in \F_3 \text{ and } 
a_6^*/a_{11}a_{12}a_{13} \in S\},$$
where
$$\chi_{e_{1, 3, 4, 7}, e_2}^{a_{11}, a_{12}, a_{13}, a_6^* }:=\Ind_{X'W_2X_6YZX_{\cN}}^{U(q)} \Inf_{V(q)}^{X'W_2X_6YZX_{\cN}}(\lambda^{a_{11},a_{12}, a_{13}, a_6^*}
\otimes \mu_{W_1}^{e_{1, 3, 4, 7}} \otimes\mu_{W_2}^{e_2}
),$$
is a family of $9(q-1)^4/2$ irreducible characters of $U(q)$ of degree $q^4/3$. 

By \cite[Section 4]{GLMP16}, this family consists of all irreducible characters of $U(q)$ 
of degree $q^n/3$ for some $n \ge 0$. \hfill\ensuremath{\square}

\vspace{2mm}

\begin{figure}[h]
\begin{center}
\begin{tikzpicture}[transform shape, place/.style={circle,draw=black,fill=black, tiny},middlearrow/.style={
    decoration={markings,
      mark=at position 0.6 with
      {\draw (0:0mm) -- +(+135:\DynkinArrowLength); \draw (0:0mm) -- +(-135:\DynkinArrowLength);},
    },
    postaction={decorate}
  }, dedge/.style={
    middlearrow,
    double distance=0.5mm,
  }]
  \node (a) at (-4,0) [circle, draw, thick, fill=none, inner sep=2pt,label=below:$\alpha_1$] {};
  \node (b) at (0,2) [circle, draw, thick, fill=none, inner sep=2pt,label=above:$\alpha_2$] {};
    \node (c) at (-2,0) [circle, draw, thick, fill=none, inner sep=2pt,label=below:$\alpha_3$] {};
  \node(d) at (0,0) [circle, draw, thick, fill=none, inner sep=2pt,label=below:$\alpha_4$] {};
  \node(e) at (2,0) [circle, draw, thick, fill=none, inner sep=2pt,label=below:$\alpha_5$]  {};
  \node(f) at (4,0) [circle, draw, thick, fill=none, inner sep=2pt,label=below:$\alpha_6$]  {};
  \node(g) at (6,0) [circle, draw, thick, fill=none, inner sep=2pt,label=below:$\alpha_7$] {};
  \node(h) at (8,0) [circle, draw, thick, fill=none, inner sep=2pt,label=below:$\alpha_8$] {};
  \draw (a) -- (c);
  \draw (b) -- (d) ;
  \draw (c) -- (d) ;
    \draw (d) -- (e) ;
      \draw (e) -- (f) ;
        \draw[dashed] (f) -- (g) ;
         \draw[dotted] (g) -- (h) ;
\end{tikzpicture}
\caption{The Dynkin diagram of $\rE_k$ for $k=6, 7, 8$. Simple roots are labelled as in CHEVIE.}
\label{tab:E6}
\end{center}
\end{figure}
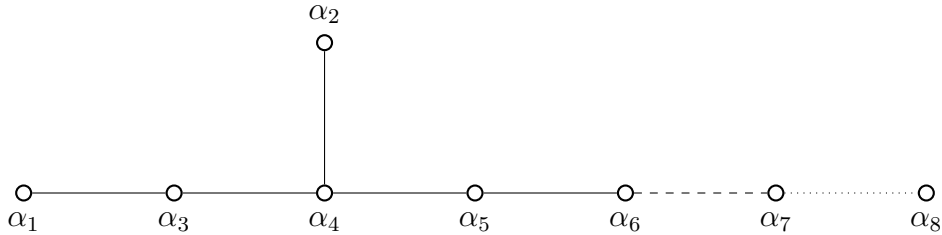

We are left with the exceptional groups of types 
$\rE_6$, $\rE_7$ and $\rE_8$. Let us first consider $p=2$. 
We define 
$$\cN_k:= \{\alpha \in \Phi^+ \mid \alpha_1 \le \alpha\} \cup \bigcup_{i=6}^k\{\alpha \in \Phi^+ \mid \alpha_k \le \alpha\},$$
for $k=6, 7, 8$. If $\Phi^+$ is a root system in type $\rE_k$, then 
$\cN_k \teq \Phi^+$, and $U(\rE_k(q))/X_{\cN_k} \cong U(\rD_4(q))$. We 
obtain a family $\cF_k$ of $4(q-1)^4$ characters of degree $q^3/2$ of $U(\rE_k(q))$ by inflation 
from $U(\rE_k(q))/X_{\cN_k}$, for $k=6, 7, 8$. 

We summarize in the following proposition the result obtained by the first and the second author in \cite{LM15} 
for $p=3$. 

\begin{prop}[\cite{LM15}, Section $3$]\label{prop:E6} Let $q=3^f$, and let $G(q)=\rE_6(q)$. There exist exactly 
$9(q-1)^6/2$ irreducible characters of $U(q)$ of degree $q^7/3$.
\end{prop}

\noindent
\emph{Construction.} Let $\cN=\{\alpha_{22}, \dots, \alpha_{36}\}$. Notice that 
$X_{\cN} \teq U(q)$ and $Z:=Z(U(q)/X_{\cN}) \cong X_{\{17, \dots, 21\}}$. 
For every $t, r, s \in \F_q$ we let 
$$x_{8, 9, 10}(t):=x_8(-t)x_9(t)x_{10}(t),$$ 
$$x_{1, 2, 3, 5, 6, 7, 11}(t, r, s):=
x_2(t)x_1(t)x_3(-t)x_5(t)x_6(-t)x_7(r)x_{11}(s),$$
and for every square $a_{8, 9, 10}^*$ in $\F_q^\times$ and a fixed a square root $e$ of $a_{8, 9, 10}^*$, let 
$$X_{8, 9, 10}:=\{x_{8, 9, 10}(t) \mid t \in \F_q\}$$ 
and 
$$F_2:=\{1, x_{2, 1, 3, 5, 6, 7, 11}(\pm e, 2e^2, 2e^2)\}, \qquad
F_4:=\{1, x_4(\pm 1/e^3)\}.$$

Let $V(q):=ZX_{8, 9, 10}F_2F_4$.  For $e_1=e_{1, 2, 3, 5, 6, 7, 11}$ and 
$e_2=e_4$ in $\F_3$, we denote by $\lambda_{e_1, e_2}^{a_{8, 9, 10}^*} \in \Irr(V(q))$ the character such that 
$$\lambda_{e_1, e_2}^{a_{8, 9, 10}^*}(x_{8, 9, 10}(t))=\phi(a_{8, 9, 10}^*t), \qquad 
\lambda_{e_1, e_2}^{a_{8, 9, 10}^*}|_{F_2}=\mu_{F_2}^{e_1}, \qquad 
\lambda_{e_1, e_2}^{a_{8, 9, 10}^*}|_{F_4}=\mu_{F_4}^{e_2},$$
and $\lambda(x_i(t))=\phi(t)$ for $i=17, \dots, 21$, and we let $H=ZX_{\{8, 9, 10\}} X_{\{12,  \dots 21\}}X_{\cN}$. Then 
we have that 
$$\cF:=\{\chi_{e_1, e_2}^{a_{8, 9, 10}^*} \mid  e_1, e_2 \in \F_3 \text{ and } a_{8, 9, 10}^* \text{ is a square in }\F_q^\times\},$$
with
$$\chi_{e_1, e_2}^{a_{8, 9, 10}^*}:=\Ind_{HX_4F_2}^{U(q)} \Inf_{V(q)}^{HX_4F_2}(\lambda_{e_1, e_2}^{a_{8, 9, 10}^*}
)$$
is a family of $9(q-1)/2$ irreducible 
characters of degree $q^7/3$. 

A split maximal torus of $G(q)$ acts transitively on $\Irr(X_{17})^\times \times \cdots \times \Irr(X_{21})^\times$; 
here $\Irr(X_i)^\times$ denotes $\Irr(X_i)\setminus\{1_{X_i}\}$. 
In particular, if $\lambda_{e_1, e_2}^{a_{17}, \dots, a_{21}, a_{8, 9, 10}^*} \in \Irr(V(q))$ is defined such that 
$$\lambda_{e_1, e_2}^{a_{17}, \dots, a_{21}, a_{8, 9, 10}^*}|_{X_{8, 9, 10}F_2F_4}
=\lambda_{e_1, e_2}^{a_{8, 9, 10}^*}|_{X_{8, 9, 10}F_2F_4}$$
and $\lambda(x_i(t))=\phi(a_it)$ for 
$i=17, \dots, 21$, then we obtain the family
$$\cF'=\{\chi_{e_1, e_2}^{a_{17}, \dots, a_{21}, a_{8, 9, 10}^*} 
\mid e_1, e_2 \in \F_3, a_{17}, \dots, a_{21} \in \F_q^\times, a_{8, 9, 10}^* \text{ is a square in }\F_q^\times\},$$ 
where 
$$\chi_{e_1, e_2}^{a_{17}, \dots, a_{21},a_{8, 9, 10}^*}:=\Ind_{HX_4F_2}^{U(q)} \Inf_{V(q)}^{HX_4F_2}(\lambda_{e_1, e_2}^{a_{17}, \dots, a_{21},a_{8, 9, 10}^*}
),$$
which consists of $9(q-1)^6/2$ elements of $\Irr(U(q))$ of degree $q^7/3$.

Finally, by \cite{LMP17$^+$} the family $\cF'$ consists of all irreducible characters of $U(q)$ 
of degree $q^n/3$ for some $n \ge 0$. \hfill\ensuremath{\square}

\vspace{2mm}

The construction in Proposition \ref{prop:E6} allows us 
to produce characters of degree $q^7/3$ also in $U(\rE_7(3^f))$ and 
$U(\rE_8(3^f))$. Let $k \in \{7, 8\}$, and let 
$$\cN_k:=\bigcup_{i=7}^k \{\alpha \in \Phi^+ \mid \alpha_k \le \alpha\}.$$
Then $\cN_k \teq \Phi^+$, and we inflate the family of $9(q-1)^6/2$ 
irreducible characters of $U(\rE_k(q))/X_{\cN_k} \cong U(\rE_6(q))$ obtained in Proposition \ref{prop:E6} 
to $U(\rE_k(q))$.

We finish with a construction of irreducible characters of 
degree $q^{16}/5$ in $U(\rE_8(5^f))$. 

\begin{prop}[\cite{LM15}, Section $4$] Let $q=5^f$, and 
let $G(q)=\rE_8(q)$. Then there exist at least $25(q-1)^8/4$ irreducible 
characters of $U(q)$ of degree $q^{16}/5$. 
\end{prop}

\noindent 
\emph{Construction.} The set $\cN:=\{\alpha_{44}, \dots, \alpha_{120}\}$ is a normal subset 
of $\Phi^+$, and $Z:=U(q)/X_{\cN} \cong X_{\{37, \dots 43\}}$. Fix $a_{37, \dots, 43}^* \in \F_q^\times$ such that 
$a_{37, \dots, 43}^*=e^4$ for 
some $e \in \F_q^\times$. Observe that such an element $a_{37, \dots, 43}^* \in \F_q^\times$ can take $(q-1)/4$ distinct values in $\F_q^\times$. 
For every $u_1, u_2, u_3 \in \F_q$, we let 
\begin{align*}
l_1(u_1)&:=x_1(u_1)x_2(2u_1)x_3(-2u_1)x_4(u_1)x_6(u_1)x_7(2u_1)x_8(-2u_1),\\
l_2(u_2)&:=l_1(u_2)x_9(u_2^2)x_{10}(-u_2^2)x_{11}(u_2^2)
x_{14}(-u_2^2)x_{15}(2u_2^2), \\
l_3(u_3)&:=l_2(u_3)x_{16}(4u_3^3)x_{17}(2u_3^3)x_{22}(3u_3^3),
\end{align*}
and we define 
$X_{12, 13}:=\{x_{12}(t)x_{13}(-t) \mid t \in \F_q\}$ 
and
$$F_4:=\{l_3(eu)x_{23}(3e^4u^4) \mid u \in \F_5\},
\qquad 
F_5:=\{x_5(v/e^5) \mid e \in \F_5\}.$$
We put $V(q):=ZX_{12, 13}F_4F_5$, and for $f_1=f_{1 \dots, 4, 6, \dots, 11, 14, \dots, 17, 22, 23}$ and $f_2=f_5$ in $\F_5$, 
we denote by 
$\lambda=\lambda_{f_1, f_2}^{a_{37, \dots, 43}^*}$ the irreducible character of $V(q)$ such that 
$$\lambda(x_{12}(t)x_{13}(-t))=\phi(a_{37, \dots, 43}^*t), 
\qquad \lambda|_{F_4}=\mu_{F_4}^{f_1}, \qquad \lambda|_{F_5}=\mu_{F_5}^{f_2},$$
and $\lambda(x_i(t))=\phi(t)$ for $i=37, \dots, 43$. 

Let $H:=ZX_{\{12, 13\}}X_{\{18, \dots, 21\}}X_{\{24, \dots, 36\}}X_{\cN}$. Then we 
have a family 
$$\cF:=\{\chi_{f_1, f_2}^{a_{37, \dots, 43}^*} \mid 
f_1, f_2 \in \F_5 \text{ and } a_{37, \dots, 43}^* \text{ is a fourth power in }\F_q^\times\},$$
where
$$\chi_{f_1, f_2}^{a_{37, \dots, 43}^*}:=\Ind_{HX_5F_4}^{U(q)}\Ind_{V(q)}^{HX_5F_4} \lambda_{f_1, f_2}^{a_{37, \dots, 43}^*},$$
of $25(q-1)/4$ irreducible characters of degree 
$q^{16}/5$. 

In a similar way of Proposition \ref{prop:E6}, we observe that a split maximal 
torus of $G(q)$ acts transitively on $\Irr(X_{37})^\times \times \cdots \times \Irr(X_{43})^\times$. 
This gives $(q-1)^7 \cdot 25(q-1)/4=25(q-1)^8/4$ irreducible characters of $U(q)$ of 
degree $q^{16}/5$. \hfill\ensuremath{\square}


\vspace{2mm}



\begin{thebibliography}{99}

\bibitem[BD06]{BD06} M. Boyarchenko and V. Drinfeld, \emph{A motivated introduction to character sheaves and the orbit method for unipotent groups in positive characteristic}, arXiv:math/0609769 (2006).


 \bibitem[CHEVIE]{CHEVIE}
M. Geck, G. Hiss, F. L\"ubeck, G. Malle and G. Pfeiffer, {\em CHEVIE -- A system for computing
and processing generic character tables for finite groups of Lie type, Weyl groups and Hecke
algebras}, Appl. Algebra Engrg. Comm. Comput. {\bf 7} (1996), 175--210.

\bibitem[DL76]{DL76} P. Deligne, G. Lusztig, \emph{Representations of reductive groups over finite fields}, Annals of Math. \textbf{103} (1976), 103--161.

\bibitem[GLMP16]{GLMP16} S. M. Goodwin, T. Le, K. Magaard and A. Paolini, \emph{Constructing characters of Sylow p-subgroups of finite Chevalley groups}, J. Algebra \textbf{468} (2016), 395--439. 

\bibitem[GMP01]{GMP01} R. Gow, M. Marjoram and A. Previtali, \emph{On the irreducible characters of a Sylow 2-subgroup of the finite symplectic group in characteristic 2}, J. Algebra {\bf 241} (2001), no.\ 1, 393--409.

\bibitem[GMR15]{GMR15} S. M. Goodwin, P. Mosch and G. R\"ohrle, \emph{On the coadjoint orbits of maximal unipotent subgroups of reductive groups}, Transformation Groups (2015), 1--28.

\bibitem[GP92]{GP92} M. Geck, and G. Pfeiffer. \emph{Unipotent characters of the Chevalley groups $D_4(q)$, $q$ odd}, Manuscripta Mathematica \textbf{76} (1992), no.\ 1, 281--304.

\bibitem[HLM11]{HLM11} F. Himstedt, T. Le and K. Magaard, {\em Characters of the Sylow $p$--subgroups of the Chevalley groups $D_4(p^n) $}, J.\ Algebra, {\bf 332} (2011), no.\ 1, 414--427.

\bibitem[HLM16]{HLM16} F. Himstedt, T. Le, K. Magaard, \emph{On the characters of the Sylow $p$-subgroups of untwisted Chevalley groups $\rY_n(p^a)$}, LMS J. Comput. Math. \textbf{19} (2016), 303--359.

\bibitem[Hup]{Hup} B. Huppert, \emph{Character Theory of Finite Groups}, Walter de Gruyter, Berlin, 1998.

\bibitem[Is]{Is} I. M. Isaacs, \emph{Character theory of finite groups}, Dover Books on Mathematics, New York, 1994.

\bibitem[Is95]{Is95} I. M. Isaacs, \emph{Characters of groups associated with finite algebras}, J. Algebra \textbf{177} (1995), 708--730.

\bibitem[Is07]{Is07} I. M. Isaacs, \emph{Counting characters of upper triangular groups}, J. Algebra \textbf{315} (2007), 698--719.

\bibitem[LM15]{LM15} T. Le and K. Magaard, {\em On the character degrees of Sylow $p$-subgroups of Chevalley groups $G(p^f)$ of type $E$}, Forum Math.\ {\bf 27} (2015), no.\ 1, 
1--55.

\bibitem[LMP17$^+$]{LMP17$^+$} T. Le, K. Magaard and A. Paolini, \emph{The irreducible characters of the Sylow $p$-subgroups of $\mathrm{D}_6(p^f)$ and $\mathrm{E}_6(p^f)$}, in preparation.

\bibitem[LS90]{LS90} L. Lambe and B. Srinivasan, \emph{A computation of Green functions for some classical groups.} Communications in Algebra \textbf{18} (1990), no.\ 10, 3507--3545.

\bibitem[Lus03]{Lus03} G. Lusztig, \emph{Character sheaves and generalizations}, in: \emph{The unity of mathematics (In honor of
the ninetieth birthday of I.M. Gelfand, Editors: P. Etingof, V. Retakh, I. M. Singer)}, Progr. Math. \textbf{244}, 443--455, Birkh\"auser Boston, Boston, MA, 2006, arXiv: math.RT/0309134.

\bibitem[Mar99]{Mar99} M. Marjoram, \emph{Irreducible characters of a Sylow $p$-subgroup of the orthogonal group}, Communications in Algebra \textbf{27} (1999), no.\ 3, 1171--1195.

\bibitem[San03]{San03} J. Sangroniz, \emph{Character degrees of the Sylow $p$-subgroups of classical groups}, Groups St. Andrews 2001 in
Oxford, vol.\ 2, 487--493, London Math. Soc. Lecture Note Ser. \textbf{305}, Cambridge Univ. Press, Cambridge,
2003.

\bibitem[Sze03]{Sze03} B. Szegedy, \emph{Characters of the Borel and Sylow subgroups of classical groups}, Journal of Algebra \textbf{267} (2003), no.\ 1, 130--136.

\end{thebibliography}
\end{document}